\documentclass[12pt]{article}
\usepackage{amssymb}
\usepackage{amsthm}
\usepackage{latexsym}

\usepackage[latin2]{inputenc}
\usepackage[leqno]{amsmath}
\usepackage{amsfonts,amssymb,latexsym,exscale}
\usepackage{graphicx}
\usepackage{pgf}
\usepackage{xspace}
\usepackage{hyperref}
\usepackage{floatflt}
\usepackage[absolute]{textpos}

\usepackage{txfonts}

\newtheorem{thm}{Theorem}[section]
\newtheorem*{thmA}{Theorem A}
\newtheorem*{thmB}{Theorem B}

\newtheorem{lem}[thm]{Lemma}
\newtheorem{defn}[thm]{Definition}
\newtheorem{lemma}[thm]{Lemma}
\newtheorem{prop}[thm]{Proposition}

\theoremstyle{remark}
\newtheorem{rem}[thm]{Remark}
\newtheorem*{ack}{Acknowledgements}
\DeclareMathOperator{\SL}{SL}
\DeclareMathOperator{\Lip}{Lip}

\DeclareMathOperator{\diam}{diam}

\DeclareMathOperator{\closure}{closure}

\DeclareMathOperator{\vol}{vol}

\DeclareMathOperator{\rank}{rank}
\DeclareMathOperator{\Ad}{Ad}
\newcommand{\from}{\colon}
\newcommand{\dimAN}{\dim_{\text{AN}}}

\newcommand{\R}{\ensuremath{\mathbb{R}}}

\DeclareMathOperator{\inter}{int}
\newcommand{\cSc}{\ensuremath{\mathcal{S}}}
\newcommand{\cSr}{\ensuremath{\Sigma}}
\newcommand{\fa}{\ensuremath{\mathfrak{a}}}

\newcommand{\fc}{\ensuremath{\mathfrak{c}}}
\newcommand{\fd}{\ensuremath{\mathfrak{d}}}
\newcommand{\fe}{\ensuremath{\mathfrak{e}}}
\newcommand{\fg}{\ensuremath{\mathfrak{g}}}
\newcommand{\fn}{\ensuremath{\mathfrak{n}}}

\newcommand{\xop}{{X_\infty^0}}

\newcommand{\V}{\ensuremath{\mathcal{V}}}

\begin{document}

\title{The distortion dimension of $\mathbb Q$--rank $1$ lattices}
\author{Enrico Leuzinger \&  Robert Young}
\date{\today}
\maketitle

\begin{abstract} Let $X=G/K$ be a symmetric space of
  noncompact type and rank $k\ge 2$. We prove that horospheres in $X$ are
  Lipschitz $(k-2)$--connected if their centers are not contained in a
  proper join factor of the spherical building of $X$ at infinity. As
  a consequence, the distortion dimension of an irreducible
  $\mathbb{Q}$--rank-$1$ lattice $\Gamma$ in a linear, semisimple Lie
  group $G$ of $\mathbb R$--rank $k$ is $k-1$.  That is, given
  $m< k-1$, a Lipschitz $m$--sphere $S$ in (a polyhedral complex
  quasi-isometric to) $\Gamma$, and a $(m+1)$--ball $B$ in $X$ (or
  $G$) filling $S$, there is a $(m+1)$--ball $B'$ in $\Gamma$ filling
  $S$ such that $\vol B'\sim \vol B$.  In particular, such arithmetic
  lattices satisfy Euclidean isoperimetric inequalities up to
  dimension $k-1$.
\end{abstract}

\vspace{6ex}



\maketitle


\section{Introduction and main results}

\vspace{2ex}

Let $G$ be a Lie group   equipped with a left invariant metric and $\Gamma\subset G$ a finitely generated discrete subgroup
equipped with a word metric. If $\Gamma$ is cocompact, then 
$\Gamma$ is quasi-isometric to $G$ and thus both have the same large-scale geometry. If $\Gamma$ is not cocompact, the large-scale geometric
 properties of $G$ and $\Gamma$ can be very 
different. For instance, $\Gamma=\SL(2,\mathbb Z)$ is exponentially
distorted in $\SL(2,\mathbb R)$, see e.g.\ \cite[Ch.\ 3]{Gro}.

Consider now a group $G$ of the form $G=\prod_{i=1}^m G_i(k_i)$, where
the $k_i$ are locally compact, non-discrete fields and the $G_i$ are
connected, absolutely almost simple algebraic groups defined over
$k_i$.  Let $\Gamma$ be an irreducible lattice in $G$.  In a
remarkable paper \cite{LMR}, Lubotzky, Mozes and Raghunathan  proved
that if the total rank $k$ of $G$, i.e.\ 
$k=\sum_{i=1}^m k_i\textup{-rank} (G_i)$, is at least $2$, then
$\Gamma$ is undistorted in $G$.  That is, the word metric $d_{\Gamma}$
on $\Gamma$ is Lipschitz equivalent to the restriction of a left
invariant metric $d_G$ on $G$ to $\Gamma$.

In \cite{BW},  Bux and Wortman conjectured a far reaching generalization of this result. In order to formulate it,
let $X$ be the product of irreducible symmetric spaces and Euclidean buildings on which $\Gamma$ acts. The total rank of $G$
is then equal to the maximal dimension of an isometrically embedded
Euclidean space in $X$, which we call the \emph{geometric rank} of $X$
and denote by $\textup{geo-rank}(X)$.
For some point $x\in X$ and a real number $r$ define the the following thickening of the orbit $\Gamma\cdot x$ in $X$
$$
X(r):=\{y\in X\mid d(y,\Gamma\cdot x)\leq r\}.
$$
Note that by the Milnor--Švarc lemma, the induced inner metric on $X(r)$ is quasi-isometric to $(\Gamma, d_{\Gamma})$. Following \cite{BW} we define $\Gamma$ as being \emph{undistorted up to dimension $m$} if:
given any $r\geq 0$, there exist real numbers $r'\geq r,\lambda\geq 1$, and $C\geq 0$ such that for any $k<m$ and any Lipschitz
$k$--sphere $S\subset X(r)$, there is a Lipschitz $(k+1)$--ball $B_{\Gamma}\subset  X(r')$ with 
$\partial B_{\Gamma}=S$ and 
$$
\textup{volume} (B_{\Gamma})\leq \lambda \ \textup{volume} (B_X) + C
$$
for all Lipschitz $(k+1)$--balls $B_X\subset X$ with $\partial B_X=S$. The \emph{distortion dimension of $\Gamma$} is then defined 
as 
$$\textup{dis-dim}(\Gamma)=\max\{m\mid \Gamma \  \textup{is undistorted up to dimension}\ m\}.
$$
The conjecture of Bux and Wortman posits that $\textup{dis-dim}(\Gamma)=\textup{geo-rank}(X)-1$.
See  \cite{BEW}, \cite{Yo} for recent progress on this conjecture.

The chief goal of the present paper is to prove the
Bux--Wortman conjecture for $\mathbb{Q}\textup{--rank}\ 1$ arithmetic groups  in
linear, semisimple groups defined over number fields, i.e.\ finite extensions
of $\mathbb Q$. For such lattices the space $X$ above is a
symmetric space  of noncompact type; that is,  there are no building
factors.  In our proof, it will be convenient to replace the subset $X(r)$ by the complement of a countable union
 of horoballs in $X$
(see \cite{Le1}, Thm. 3.6).  Like $X(r)$, this is quasi-isometric
to  $(\Gamma, d_{\Gamma})$.  A crucial fact is that for $\mathbb{Q}\textup{--rank}\ 1$
lattices these horospheres  are \emph{disjoint}.  

\begin{thmA}[Distortion dimension]
  The distortion dimension of an irreducible
  $\mathbb{Q}\textup{--rank}$ $1$ lattice in a linear, semisimple Lie
  group of $\mathbb R$--\textup{rank} $k$ is $k-1$.  If $k\ge 2$, then
  such an arithmetic lattice satisfies Euclidean isoperimetric
  inequalities up to dimension $k-1$ and an exponential isoperimetric
  inequality in dimension $k$.

  That is, there is a $(k-2)$--connected complex $Y$ that is
  equivariantly quasi-isometric to $\Gamma$ such that for any
  $m\le k-2$ Lipschitz $m$--sphere $S\subset Y$, there is a Lipschitz
  $m$--ball $B\subset Y$ such that $\partial B=S$ and
  $$\vol B\lesssim (\vol S)^{m+1/m}.$$

  Conversely, for $r>1$, there is a Lipschitz sphere
  $S\from S^{k-1}\to Y$ such that $\vol S\sim r^{k-1}$ but
  $\vol B\gtrsim e^r$ for any Lipschitz $k$--ball $B\subset Y$ such
  that $\partial B=S$.
\end{thmA}
Gromov showed that any nonuniform lattice $\Gamma$ in a semisimple
group $G$ of $\R$--rank 1 is exponentially distorted and thus has
distortion dimension 0 (see \cite[3.G]{Gro}).

In many cases, the lower bound in Theorem A follows from work of
Wortman \cite{Wo}, who proved that arithmetic subgroups of semisimple
groups of relative $\mathbb Q$--type $A_n, B_n, C_n, D_n, E_6$ or
$ E_7$ are exponentially distorted in dimension
$\textup{geo-rank}(X)$.

By work of Young \cite{Yo}, 
non-distortion for subsets of spaces with finite Assouad--Nagata
dimension is a consequence of Lipschitz connectivity.  Recall that $Z$ is Lipschitz $n$--connected if for all
$d\le n$ and any Lipschitz map $\alpha\from S^d\to Z$, there is an
extension $\beta\from D^{d+1}\to Z$ such that
$\Lip(\beta)\lesssim\Lip(\alpha)$.  Theorem~1.3 of \cite{Yo} (see also \cite{Yo1}) states
the following:
\begin{prop}[Distortion and connectivity, {\cite[1.3]{Yo}}]\label{prop:LipUndistorted}
  Let $X$ be a metric space and let $Z\subset X$ be a nonempty
  closed subset with inner metric induced by the metric
  of $X$.  Suppose in addition that $X$ is a geodesic metric space
  such that the Assouad--Nagata dimension $\dimAN(X)$ of $X$ is
  finite and one of the following is true:
  \begin{itemize}
  \item $Z$ is Lipschitz $n$--connected.
  \item $X$ is Lipschitz $n$--connected, and if $X_p,p\in P$ are the connected components of
    $X\smallsetminus Z$, then the sets $H_p=\partial X_p$ are
    Lipschitz $n$--connected with uniformly bounded
    implicit constant.
  \end{itemize}
  Then $Z$ is undistorted up to dimension $n+1$.  
\end{prop}

The upper bounds in Theorem A follow from this proposition and the
following Theorem B.  The lower bound will follow from
Prop.~\ref{prop:lowerBoundsHorospheres}, which shows that the
intersection of a horosphere with a flat can produce a $(k-1)$--sphere
with exponentially large filling volume.

\begin{thmB}[Horospheres are highly Lipschitz connected]
  Let $X=G/K$ be a symmetric space of noncompact type and rank
  $k$. Then any horosphere in $X$ whose center is not contained in a proper 
  join factor of the boundary  of $X$ at infinity is
  Lipschitz $(k-2)$--connected and thus undistorted up to dimension $k-1$. 
\end{thmB}
\begin{rem} Theorem B generalizes work of C. Dru\c{t}u, who proved
  non-distortion of horospheres up to dimension $2$ \cite{Dr1,Dr2}.
\end{rem}

\begin{ack}
  The research leading to this paper occurred while the authors were
  visiting the Forschungsinstitut für Mathematik at ETH Zürich. 
   We would like to  thank FIM and especially Urs Lang for their hospitality.
\end{ack}

\section{Sketch of proof}
The core of this paper is the proof of Theorem~B, the Lipschitz
connectivity of a horosphere $Z$ which bounds a horoball $H$ in the
symmetric space $X$.  This proof is structured similarly to the proof
of Theorem~4.1 in \cite{Yo}, which established Lipschitz connectivity for
certain horospheres in Euclidean buildings.
The proof of Theorem~4.1 in \cite{Yo} used a version of Morse theory
based on the \emph{downward link at infinity} of a vertex in a Euclidean
building.  Let $\mathcal B$ be such a building and let $\mathcal B_\infty$ be its boundary at
infinity.  Then, if a horoball is centered at a point $\tau\in \mathcal B_\infty$, the
downward link at infinity of a point $x\in \mathcal B$ is the subset of
$\mathcal B_\infty$ consisting of the limit points of geodesic rays from $x$
that point away from $\tau$.

We start by constructing an analogue of the downward link for
symmetric spaces.  When $X$ is a symmetric space, we let $X_\infty$ be
its (geodesic) boundary at infinity and equip $X_\infty$ with the Tits
metric associated to the angular metric $\angle$.  Then $X_\infty$ has
the additional structure of a spherical building, see \cite{BGS}, Appendix 5
and \cite{BH}, Ch. II.10.  We use the building structure to
replace downward links by \emph{shadows}.  If $H$ is a horoball and $x\in H$, then each
chamber in $X_\infty$ is the limit set of a Weyl chamber based at
$x$.  The shadow of $x$ consists of the limit sets of Weyl chambers
that ``strongly point out of'' $H$ (Sec.~\ref{sec:shadows}).  We
denote the shadow of $x$ by $\cSc_x$; note that this is a set of chambers in $X_\infty$
viewed as a spherical building.

For $x\in H$ we denote the union of the chambers in $\cSc_x$ by
$\cSr_x$.  This is a subset of the
geodesic boundary $X_\infty$.  It is a collection of directions (or limits) of geodesic rays starting at $x$ that head toward
$Z=\partial H$ quickly.  In particular, there is a Lipschitz map
$i_x\from \cSr_x\to Z$ that takes each direction to the point where
the corresponding ray intersects $Z$.

Next, we show that shadows are $(k-2)$--connected
(Sections~\ref{sec:oppositeFlats} and \ref{sec:infConnect}) by
constructing many flats inside each shadow.  For each shadow $\cSc_x$,
we will find a chamber $\fd\subset X_\infty$ such that for each
chamber $\fe\in \cSc_x$, there is a flat $E_{\fd,\fe}\subset X$ that
is spanned by $\fd$ and $\fe$ and whose boundary at infinity is
contained in a larger shadow.  The boundary at infinity of the union
of these flats is a union of apartments in the spherical building.
This union has the homotopy type of a wedge of $(k-1)$--spheres, so
$\cSr_x$ is $(k-2)$--connected inside a larger shadow.

Finally, we use shadows to prove that $Z$ is Lipschitz
$(k-2)$--connected (Section~\ref{sec:LipConnZ}). Let $\Delta_Z$ be the
infinite-dimensional simplex with vertex set $Z$.  It suffices to
construct a map $\Omega\from \Delta_Z\to Z$ with certain metric
properties. Since shadows are $(k-2)$--connected, we construct a map
$\Omega_\infty\from \Delta_Z\to X_\infty$ such that the image of each
face of $\Delta_Z$ lies in a shadow.  We can then use the map $i_x$
above to send these shadows to $Z$.

\section{Proof of Theorem B}

\subsection{Preliminaries and standing assumptions}
Let $X=G/K$ be a symmetric space of noncompact type and rank $k\ge 2$.
If $g\in G$, let $[g]=gK\in X$ be the corresponding coset of $K$.  Let
$X_\infty$ be the (geodesic) boundary of $X$ at infinity and equip $X_\infty$
with the Tits metric associated to the angular metric $\angle$.  Note
that if $X$ is a Riemannian product of irreducible factors,
$X=X_1\times\ldots\times X_m$, then its boundary $X_\infty$ is the
spherical join of the boundaries of the factors,
$X_\infty=(X_1)_\infty\ast\cdots\ast(X_m)_\infty$ (see \cite{BH},
II.8.11.).
 
Let $\tau\in X_\infty$ be a point that is not contained in a proper
join factor of $X_\infty$, i.e.,  $\tau$ is the limit point of a
geodesic ray that is not constant on any proper factor.  Let
$H\subset X$ be a horoball centered at $\tau$ and let $Z:=\partial H$ be the boundary horosphere.
Let $h\from X\to \R$ be the Busemann function centered at $\tau$,
oriented so that $H=h^{-1}([0,\infty))$ and $Z=h^{-1}(\{0\})$.  We define $H_\infty$ to be
the open ball $B_{\pi/2}(\tau)$ in $X_\infty$, so that
$(h\circ r)'(t)>0$ for any geodesic ray $r$ that is asymptotic to a
point in $H_\infty$.  Since $\tau$ is not contained in a proper join
factor, there is a chamber $\fc \subset X_\infty$ such that
$\tau\in \fc$ and $\fc\subset H_\infty$ (see \cite{KL}, Section 3). 
 This is the main reason for the assumption that $\tau$ is not contained in a proper join
factor; the fact that $\fc\subset H_\infty$ is crucial to Lemma
\ref{lem:lipschitzPushing}. 
We note that if $X$ is irreducible, a Weyl chamber has diameter less than $\pi/2$. Thus any $\tau\in X_\infty$ and
the associated Busemann function also have the above properties.

The stabilizer of $\fc$ in $G$ is a minimal parabolic subgroup $P$,
and the set of (maximal) chambers in $X_\infty$ can be identified with
the homogeneous space $G/P$ (see e.g. \cite[Ch. 1.2]{Wa}, or
\cite[Lemma 4.1]{Mo}).  We let $P=NAM$ be its Levi decomposition. Thus
$N$ is nilpotent, $A$ abelian, and $M$ is the centralizer of $A$ in
$K$ and in particular compact.  Note that, by the Iwasawa
decomposition $G=NAK$, $P$ acts transitively on $X$.

For $x\in X$ there is a unique flat $E_x$ containing $x$ and $\fc$.
Let $E_0= E_{eK}=[A]$ and let $\fc^*$ be the chamber opposite to $\fc$
in $E_0$.  More generally, let $\fc^*_x$ be the chamber opposite to
$\fc$ in $E_x$.  If $x=[p]$ and $p=nam$ for some $n\in N$, $a\in A, m\in M$,
then $E_x=nE_0=pE_0$ and $c^*_x=p\fc^*=n\fc^*$.

Let $\xop(\fc)\subset X_\infty$ be the set of chambers opposite to
$\fc$, so that $\xop(\fc)=N\cdot \fc^*$.  If $\fd\subset \xop(\fc)$,
we define $E_\fd$ to be the flat asymptotic to both $\fc$ and $\fd$.

The notation $f \lesssim g$ means that there is some constant $c$ such
that $f\le c g$.  We write $f\sim g$ if there is some $c>0$ such that
$c^{-1}g\le f\le c g$.  When $c$ depends on $x$ and $y$, we write
$f \lesssim_{x,y} g$ or $f\sim_{x,y} g$.  All of our constants will
depend implicitly on $X$ and $\tau$, so we will omit $X$ and $\tau$
from these subscripts.

\subsection{Shadows}\label{sec:shadows}

In this section, we define the shadow $\cSc_{x}\subset \xop(\fc)$ of a
point $x\in X$ and prove some of its properties.  The shadow of a
point in a symmetric space will play a similar role to the downward
link at infinity of a point in a Euclidean building in \cite{Yo}.

\begin{defn}
  Let $\fg$ and $\fn$ be the Lie algebras of $G$ and $N$.  The metric
  on $X$ is induced by an $Ad(K)$--invariant norm $\|\cdot\|$ on $\fg$ 
  (see \cite{E}, 2.7.1).  For $n\in N$,
  define $d_N(n)=\|\log n\|$.

  The \emph{$r$--shadow} of $x=[p]$ with respect to $\fc$ is the set of
  chambers
  $$\cSc_{x}(r) := \{\fc^*_{[pn]}\subset \xop(\fc)\mid n\in N, d_N(n)<r\}.$$
  We set $\cSc_{x}=\cSc_{x}(1)$. Note that $\cSc_{x}(r)$ is well-defined, 
   since  by 
  $Ad(K)$--invariance $d_N(mnm^{-1})=d_N(n)$ for all $m\in M\subset K$ and $n\in N$;
   moreover  $\cSc_{x}(r)  =p\cdot \cSc_{[e]}(r)$.
\end{defn}
If $\fd\in \xop(\fc)$ and $x=[p]\in X$, then there is a 
$q_x(\fd)\in N$ such that $[p q_x(\fd)]\in E_\fd$.  In fact, if $n_\fd,n\in N$
and $a\in A$ are such that $E_\fd=[n_\fd A]$ and $x=[na]$, then we
can write $q_x(\fd)=a^{-1}n^{-1}n_\fd a$. This is unique up to conjugation
 by some $m\in M$.  We set
$$\rho_x(\fd)=d_N(q_x(\fd)),$$
this is well-defined (i.e., independent of the choice of $q_x(\fd)$). 
Further we have:  $\fd\in \cSc_x(r)$ if and only if $\rho_x(\fd)<r$.
Roughly, the function $\rho_x$ measures the angle at which the Weyl
chamber based at $x$ and asymptotic to $\fd$ meets the apartment
$E_x$.  When $\rho_x(\fd)$ is small, then $\fd$ is close to $E_x$, and
when it is large, $\fd$ deviates more sharply.  We think of $\cSc_{x}$
as the ``shadow'' on $\xop(\fc)$ cast by a light at $\fc$ shining on a
ball around $x$ with radius roughly 1.

\begin{lemma}\label{lem:compare}
  If $x\in X$ and $\fd\in \cSc_x(r)$, then $d(x,E_\fd)<r$.
  Conversely, there is a $c>0$ depending on $X$ such that if
  $d(x,E_\fd)<r$, then $\fd\in \cSc_{x}(e^{c r}).$
\end{lemma}
\begin{proof}
  First, if $x=[p]\in X$ and $\fd\in \cSc_x(r)$, then, as we have seen above, there is an
  $n=q_x(\fd)\in N$ such that $[pn]\in E_\fd$ and $d_N(n)<r$.  It follows that
  $d(x,E_\fd)\leq d([p],[pn]) \le d_N(n)<r$.

  Conversely, suppose that $d(x,y)<r$.  Without loss of generality, we
  may take $x=[e]$.  Let $n\in N$ and $a\in A$ be such that $y=[na]$.
  Then $\fc^*_y=\fc^*_{[n]}$, so it suffices to show that $d_N(n)$ is
  exponentially bounded in $r$.

  The map $[na]\mapsto a$ is a distance-decreasing map from $X$ to
  $A$, so $$d([e],[a])\le d([e],[na])=d(x,y)< r.$$  It follows that
  $$d([e],[n])\le d([e],[na])+d([na],[n])\le 2r.$$
  By \cite{LMR}, $n$ satisfies the
  inequality 
  $$\log d_N(n) \lesssim d([e],[n])\lesssim 1+\log d_N(n),$$
  so $d_N(n)\le e^{c r}$ as desired.  
\end{proof}

The shadow of $x$ grows exponentially as $x$ moves toward $\fc$.  
\begin{lem}\label{lem:dil}
  Let $x\in X$ and let $\gamma\from [0,\infty)\to X$ be a unit-speed
  geodesic ray starting at $x$ and pointing at a point $\sigma\in \inter \fc$
  in the interior of $\fc$.  There is a constant
  $\kappa>0$ depending on $\sigma$ such that for all $t\ge 0$ and all $\fd\in  \xop(\fc)$,
  $$\rho_{\gamma(t)}(\fd)\lesssim e^{-\kappa t}\rho_x(\fd).$$
\end{lem}

\begin{proof}
  Without loss of generality, we may assume that $x=[e]$ and that
  $x'=[\exp tV]$ for a regular unit vector $V$ in the open Weyl
  chamber in $T_{x}X$ corresponding to the chamber
  $\fc\in X_{\infty}$ (see \cite{BGS}, appendix 5). 
   Let $\Sigma_+$ be the corresponding set of
  positive roots.  Let $a(t)=\exp tV$ and let
  $$\kappa:=\min_{\alpha\in \Sigma_+}\alpha(V)>0.$$

  Let $n=q_x(\fd)$ so that $\rho_x(\fd)=d_N(n)$ and 
  $q_{\gamma(t)}(\fd)=a(-t)na(t).$
  The Lie algebra $\mathfrak{n}$ of $N$ can be written as the sum of (positive) root spaces
  $\mathfrak n=\sum_{\alpha\in\Sigma^+}\mathfrak{g}_\alpha$. Thus $\log
  n=\sum_{\alpha\in\Sigma^+}X_\alpha$ and 
  $$q_{\gamma(t)}(\fd)=\exp[\Ad(a(-t))\log n]=\exp
  \sum_{\alpha\in\Sigma^+}e^{-t\alpha(V)}X_\alpha.$$
  Then
  $$\rho_{\gamma(t)}(\fd)=d_N(q_{\gamma(t)}(\fd))\lesssim e^{-\kappa t}d_N(n).$$
  as desired.
\end{proof}

Let $A_+\subset A$ denote the Weyl chamber based at the identity that
is asymptotic to $\fc$.  
For $p\in P$ and $x=[p]\in X$, let $C_x:=[pA_+]$ be the Weyl chamber
based at $x$ and asymptotic to $\fc$.  Because shadows grow
exponentially with height, we can expand $\cSc_x$ greatly by replacing
$x$ by a point in $C_x$.  Let
$$D_x:=\{y\in X\mid d(y,C_x)< 1\}.$$
The set $D_x$ is roughly the set of points whose shadows contain
$\cSc_x$.  

\begin{lemma}\label{lem:DxShadows}
  There is a $\rho>0$ such that for all $x\in X$ and all $y\in D_x$,
  $\cSc_{x}\subset \cSc_{y}(\rho)$.
\end{lemma}
\begin{proof}
  Without loss of generality, suppose that $x=[e]$, so that
  $C_x=[A_+]$, where $A_+$ is the Weyl chamber in $A$ corresponding to
  $\fc$.  If we write $y=[an]$ with $a\in A$, $n\in N$, then
  $d(a,A^+)\le 1$ and $d_N(n)\lesssim 1$.  We write $a=a_+b$, where
  $a_+\in A^+$ and $\|\log b\|<1$.  

  Suppose that $n'\in N$ and $d_N(n')<1$, so that $\fc^*_{[n']}\in
  \cSc_x$.  Then 
  $$E_{[n']}=[n' A]=[an n^{-1}(a^{-1}n'a) A]$$
  and  $\fc^*_{[n']}\in \cSc_{y}(\rho)$ if and only if
  $d_N(n^{-1}(a^{-1}n'a))<\rho.$
  But $d_N(n)\lesssim 1$, and 
  $$d_N(a^{-1}n'a)=\|\Ad(a_+b) \log n'\|\lesssim 1$$
  because the eigenvalues of $\Ad(a_+)$ are all at most 1 and $\Ad(b)$
  is bounded.  It follows that there is a $\rho$ depending on $X$ so
  that $\fc^*_{[n']}\in \cSc_{y}(\rho)$ and
  $\cSc_{x}\subset \cSc_{u}(\rho)$.
\end{proof}

The shadows of a collection of points can all be contained in a larger shadow.
\begin{lem}\label{lem:largeShadows}
  Let $x\in X$ and let $\gamma\from [0,\infty)\to X$ be a unit-speed
  geodesic ray starting at $x$ and pointing at a point
  $\sigma \in \inter \fc$ in the interior of $\fc$.  Let $r>0$.  There is a
  point $x'=\gamma(t)$ with $t\lesssim_\sigma r+1$ such that 
  $$\bigcup_{y\in B_r(x)} \cSc_{y}\subset \cSc_{x'}$$
  and $x'\in \bigcap_{y\in B_r(x)}D_y$.
\end{lem}
\begin{proof}
  Without loss of generality, we take $x=[e]$, $a(t)=\exp tV$, and
  $\gamma(t)=[\exp tV]$ as in the proof of Lemma~\ref{lem:dil}.

  Suppose that $y\in B_r(x)$.  We claim that there is a $c>0$ such
  that $\cSc_y\subset \cSc_{\gamma(t)}$ and $\gamma(t)\in D_y$ for all
  $t\ge c r+c$.  First, we claim that $\cSc_y\subset \cSc_{\gamma(t)}$
  when $t$ is large.  If $\fd\in \cSc_y$, Lemma~\ref{lem:compare}
  implies that $d(y,E_\fd)<1$, so $d(x,E_\fd)<r + 1$.  If $c_0$ is as
  in Lemma~\ref{lem:compare}, then $\fd\in \cSc_{x}(e^{c_0 (r+1)})$, and
  by Lemma~\ref{lem:dil}, there is a $c_1$ such that
  $\fd\in \cSc_{\gamma(t)}$ for all $t>c_1(r+1)$.

  Next, we claim that $\gamma(t)\in D_y$ when $t$ is large.  Let
  $n\in N$, $a\in A$ be such that $y=[na]$ and $E_y=[nA]$.  Let
  $\tilde{\gamma}(t)=n\gamma(t)$ so that $\tilde{\gamma}$ is a
  geodesic ray toward $\sigma$ that lies in $E_y$.  Then
  $d(y,\tilde{\gamma}(0))=d([a],[e])<r$, and since $\tilde{\gamma}$
  points toward the interior of $\fc$, there is a $c_2$ such that
  $\tilde{\gamma}(t)\in C_y$ for all $t\ge c_2 r$.  

  If $t>\max\{c_2r, c_1(r+1)\}$, then
  $$d(\gamma(t), C_y)\le d(\gamma(t),\tilde{\gamma}(t))\le d_N(a(-t)na(t)).$$
  and, since $\fc^*_y\in \cSc_{\gamma(t)}$, we have
  $$d_N(a(-t)na(t))=\rho_{\gamma(t)}(\fc^*_y)<1.$$
  So $\gamma(t)\in D_y$ as desired.
\end{proof}

Finally, we can use shadows to define a map that projects directions
in $X_\infty$ to points in $Z$.
\begin{lem}\label{lem:lipschitzPushing}
  Let $H$, $h$, $Z$, $\tau$, and $\fc$ be as in the standing
  assumptions.

  For $u\in H$ such that $h(u)\ge 1$ and $\rho >0$, let
  $$\cSr_u(\rho)=\bigcup_{\fd\in \cSc_{u}(\rho)}\fd\subset X_\infty$$
  be the point set in $X_\infty$ determined by the chambers in the
  shadow $\cSc_{u}(\rho)$.  We define $i_u\from \xop(\fc)\to Z$ so
  that $i_u(\sigma)$ is the point where the geodesic ray from $u$
  toward $\sigma$ intersects $Z$.  The distance traveled before
  reaching $Z$ is bounded in terms of $h(u)$ and $\rho$:
  $$d(i_u(\sigma),u)\lesssim h(u)+\rho.$$
  The map $i_u$ is locally Lipschitz with
  $$\Lip(i_u|_{\cSr_u(\rho)})\lesssim (\rho+1)^2 d(u,Z).$$

  Furthermore, if $u_1,u_2\in H$ are such that $h(u_i)\ge 1$ and if
  $\sigma_1,\sigma_2\in \cSr_{u_1}(\rho)\cap \cSr_{u_2}(\rho)$, then
  \begin{equation}\label{eq:twoVariablePushing}
    d(i_{u_1}(\sigma_1),i_{u_2}(\sigma_2))\lesssim(\rho+1)^2\bigl(d(u_1,u_2)+\min\{h(u_1),h(u_2)\}\cdot \angle(\sigma_1,\sigma_2)\bigr).
  \end{equation}
\end{lem}
\begin{proof}
  We proceed similarly to the arguments in Section~4.5 of
  \cite{Yo}.  Let
  $$CX_\infty=(X_\infty \times [0,\infty))/(X_\infty \times 0)$$
  be the infinite cone over $X_\infty$.  We equip $CX_\infty$ with the
  Euclidean cone metric 
  $$d((\sigma_1,t_1),(\sigma_2,t_2))^2=t_1^2+t_2^2-2t_1t_2\cos \angle(\sigma_1,\sigma_2)$$
  so that the cone over an apartment in
  $X_\infty$ is isometric to Euclidean space.  For
  $\sigma\in X_\infty$, let $r_{x,\sigma}\from[0,\infty)\to X$ be the
  unit-speed geodesic ray based at $x$ that is asymptotic to $\sigma$,
  and let $e_x\from CX_\infty\to X$ be the ``exponential map''
  $$e_x(\sigma,t)=r_{x,\sigma}(t).$$
  Because $X$ is a CAT(0) space, this is a distance-decreasing map and
  if $x,x'\in X$, then
  \begin{equation}\label{eq:parallelRays}
    d(e_x(\sigma,t),e_{x'}(\sigma,t))\le d(x,x').
  \end{equation}

  Since $\tau$ is not in a proper join factor (by the standing assumptions), there is an $\epsilon>0$ such that
  $\fc\subset B_{\pi/2-\epsilon}(\tau)$.  Each point
  $\sigma\in \xop(\fc)$ is opposite to a point in $\fc$, so
  $\angle(\sigma,\tau)>\frac{\pi}{2}+\epsilon$.  By the concavity of
  $h$, it follows that for each such $\sigma$, the ray $r_{u,\sigma}$
  intersects $Z$ exactly once.

  Our first task is to show that $d(i_u(\sigma),u)\lesssim h(u)+\rho$
  for all $\sigma\in \cSr_u(\rho)$.  Let $T_u(\sigma)=d(u,i_u(\sigma))$,
  so that
  $$i_u(\sigma)=r_{u,\sigma}(T_u(\sigma)).$$
  Without loss of generality, we can take $u$ to be the basepoint
  $u=[e]$.  Let
  $$c=\max \{d([e],[n])\mid n\in N, d_N(n)\le 1\}.$$
  Then for any $n\in N$, we have $d([e],[n])\le cd_N(n)$, and if
  $\fd\in \cSc_{u}(\rho)$, then $d(u,E_\fd)\le c\rho$.

  Suppose that $\sigma\in \cSr_u(\rho)$ and that $\fd\in \cSc_{u}(\rho)$ is a
  chamber containing $\sigma$.  Let $n\in N$ be such that
  $E_\fd=[nA]$.  The geodesic $r_{[n],\sigma}$ is contained in the flat $E_\fd$,
  so
  $$h(r_{[n],\sigma}(t))=h([n])+t\cos \angle(\sigma,\tau)\le h([n])- t\sin \epsilon.$$
  Since $r_{[n],\sigma}$ and $r_{u,\sigma}$ are asymptotic to the same
  point, the distance $d(r_{u,\sigma}(t),r_{[n],\sigma}(t))$ is a
  non-increasing function of $t$, and 
  $$|h(r_{u,\sigma}(t))-h(r_{[n],\sigma}(t))|\le d(u,[n])\le c\rho$$
  for all $t\ge 0$.  It follows that 
  $$T_u(\sigma)\le \frac{h(u)+c\rho}{\sin \epsilon}\lesssim
  h(u)+\rho.$$
  as desired.  Let $b:=\frac{1+c\rho}{\sin \epsilon}$; then
  $b \lesssim \rho+1$ and $T_u(\sigma)\le b h(u)$ for all $u$ such
  that $h(u)\ge 1$.

  Next, we bound the Lipschitz constant of $T_u$.  If
  $\sigma_1, \sigma_2\in \cSr_u(\rho)$ and
  $\angle(\sigma_1,\sigma_2)\ge \frac{1}{2b},$
  then 
  $$\frac{|T_u(\sigma_1)-T_u(\sigma_2)|}{\angle(\sigma_1,\sigma_2)}\le
  \frac{b h(u)}{(2b)^{-1}}\lesssim b^2h(u).$$

  Otherwise, consider the case that
  $\angle(\sigma_1,\sigma_2)<\frac{1}{2b}.$ Let $r_i=r_{u,\sigma_i}$
  and $T_i=T_u(\sigma_i)$.  Without loss of generality, suppose
  $T_1\le T_2$, so that $h(r_1(T_1))=0$ and $h(r_2(T_1))\ge 0$.  We
  will show that $h(r_2(T_1))$ is small and that $(h\circ r_2)(t)$ is
  decreasing quickly at $t=T_1$.

  Since $X$ is CAT(0), we have
  $$d(r_1(T_1),r_2(T_1))\le b h(u) \angle(\sigma_1,\sigma_2)\le \frac{h(u)}{2}.$$
  It follows that
  $(h\circ r_2)(T_1)\le b h(u) \angle(\sigma_1,\sigma_2)$.
  Furthermore, since $(h\circ r_2)(t)$ is concave down, we have
  $$-(h\circ r_2)'(t)> \frac{h(u)-(h\circ r_2)(T_1)}{T_1}\ge\frac{h(u)}{2bh(u)} \gtrsim \frac{1}{b}$$
  for all $t\ge T_1$.  Consequently,
  \begin{align*}
    T_2-T_1 &\le \frac{(h\circ r_2)(T_1)}{-(h\circ r_2)'(T_1)}\\
            &\lesssim \frac{b h(u) \angle(\sigma_1,\sigma_2)}{b^{-1}} \\
            &\lesssim  b^2 h(u) \angle(\sigma_1,\sigma_2),
  \end{align*}
  so $\Lip T_u\lesssim b^2 h(u)$.
  
  Thus, for all $\sigma_1,\sigma_2\in \cSr_u$, if $r_i$ and $T_i$ are as
  above, we have
  \begin{align*}
    d(i_u(\sigma_1),i_u(\sigma_2))&=d(r_1(T_1),r_2(T_2))\\
    &\le d(r_1(T_1),r_2(T_1))+|T_2-T_1|\\
    &\lesssim bh(u) \angle(\sigma_1,\sigma_2)+ b^2 h(u)\angle(\sigma_1,\sigma_2)\\
    &\lesssim (\rho+1)^2 h(u)\angle(\sigma_1,\sigma_2),
  \end{align*}
  so $\Lip(i_u)\lesssim (\rho+1)^2  h(u)$.  
  
  Finally, we prove \eqref{eq:twoVariablePushing}.  Suppose that
  $u_1,u_2\in h^{-1}([1,\infty))$ and
  $\sigma\in \cSr_{u_1}(\rho)\cap \cSr_{u_2}(\rho)$.  Let
  $r_i=r_{u_i,\sigma}$ and $T_i=T_{u_i}(\sigma)$ and suppose that
  $T_1\le T_2$.  By the convexity of $h\circ r_1$, we know that
  $(h\circ r_1)'(t)<0$ for all $t\ge T_1$.  In fact,
  \begin{align*}
    (h\circ r_1)'(t)&\le \frac{(h\circ r_1)(T_1)-(h\circ r_1)(0)}{T_1}\\ 
                    &\le \frac{-h(u_1)}{bh(u_1)}= -b^{-1}
  \end{align*}
  for all $t\ge T_1$.  

  Since $X$ is a CAT(0) space, we have
  $$d(r_1(t),r_2(t))\le d(u_1,u_2)$$
  for all $t\ge0$.  Then
  \begin{align*}
    h(r_2(T_2))&\le h(r_1(T_2))+d(u_1,u_2)\\ 
               &\le-(T_2-T_1)b^{-1}+d(u_1,u_2).
  \end{align*}
  But $h(r_2(T_2))=0$, so $|T_2-T_1|\le bd(u_1,u_2)$.
  Therefore, 
  $$d(i_{u_1}(\sigma),i_{u_2}(\sigma))\le d(u_1,u_2)+|T_1-T_2|\lesssim
  b d(u_1,u_2).$$

  Thus, if
  $\sigma_1,\sigma_2\in \cSr_{u_1}(\rho)\cap \cSr_{u_2}(\rho)$, then
  $$d(i_{u_1}(\sigma_1),i_{u_2}(\sigma_2))\lesssim b d(u_1,u_2)+b^2
  \min\{h(u_1),h(u_2)\}\cdot \angle(\sigma_1,\sigma_2)$$
  as desired.
\end{proof}

\subsection{Finding opposite flats}\label{sec:oppositeFlats}
We will need to show that the shadows of points are highly
connected. To that end we will show in this section that shadows of
points contain the boundaries of many flats.  First, we claim that
$\xop$,  the set of chambers opposite to $\fc$, contains many flats.

\begin{defn}
  If $X$ is a symmetric space of rank $k$ and $\fc\subset X_\infty$ is
  a chamber, and $E\subset X$ is a $k$--flat, then we say that $E$ is
  \emph{opposite} to $\fc$ if every chamber in the boundary of $E$ at infinity, $E_\infty\subset X_\infty$, is opposite
  to $\fc$.
\end{defn}

\begin{lemma}\label{op}
  If $\fc\subset X_\infty$ is a chamber, then there is some $k$--flat
  $E$ such that $E$ is opposite to $\fc$.
\end{lemma}

\begin{proof} Let  
  the stabilizer of $\fc$ in $G$ be the minimal parabolic subgroup $P=NAM$
  and identify the set of (maximal) chambers in $X_\infty$ 
  with the homogeneous space $G/P$. Recall that $\xop(\fc)\subset X_\infty$ denotes
  the set of chambers opposite to $\fc$ and if $\fc^*$ is one such
  chamber, then $\xop(\fc)=N\cdot \fc^*$.  Under the identification of the set of chambers in $X_\infty$
  with $G/P$, we find $\xop(\fc)=N w^*P$, where $w^*$ is the longest
  element in the Weyl group of $X$. This orbit is the big cell in the
  Bruhat decomposition of $G$ and its complement has measure zero.  In
  fact, its complement has codimension at least $1$ (see \cite{Wa},
  Prop. 1.2.4.9 or \cite{He}, Ch. IX, Cor. 1.8), so we can view
  $\xop(\fc)$ as an open submanifold of $G/P$ whose complement has
  codimension at least $1$.

  Let $F$ be a maximal flat in $X$. Its boundary at infinity $F_\infty\subset X_\infty$ consists of finitely many chambers, say
$\fc_1,\fc_2,\ldots,\fc_m$. The set $\xop(\fc_i)$ of chambers that are opposite to $\fc_i$ is an open submanifold
whose complement has 
codimension at least $1$ for all $i=1,\ldots,m$. Thus the set of all chambers in $X_\infty$ simultaneously opposite to all  chambers of $F$
is the intersection $\cap_{i=1}^m \xop(\fc_i)$. This is an open
submanifold of $G/P$ whose complement has codimension at least $1$. In
particular, there
is a chamber $\fc'$ opposite to $F$. If we write the given chamber $\fc$ as $\fc=g\fc'$ for some $g\in G$, then $E=gF$ satisfies the claim of the Lemma.
\end{proof}

\begin{rem} The above proof shows that if a chamber $\fc$ is opposite
  to a fixed flat $E$, then so are all chambers  in an
  open neighborhood of $\fc$. Similarly, all flats of the form $hE$
  for $h$ in an open neighborhood of $e\in G$ are opposite to a fixed
  $\fc$.
\end{rem}

We use Lemma~\ref{op} to prove lower bounds on the filling invariants of
horospheres.
\begin{prop}\label{prop:lowerBoundsHorospheres}
  For some $c>0$ and for all sufficiently large $r$, there is a
  Lipschitz map $\alpha\from S^{k-1}\to Z$ with $\Lip(\alpha)\sim r$
  such that any Lipschitz extension $\beta\from D^{k}\to Z$
  satisfies
  $$\vol \beta\ge e^{cr}.$$
\end{prop}
\begin{proof}
  Our bound is based on an estimate of the $k$th divergence of $X$ due
  to Leuzinger \cite{LeuzingerCorank}.  Leuzinger showed that there
  are $c_0>0$ and $R>1$ such that if $F$ is a flat in $X$, $x\in F$ and $r>R$,
  then the $(k-1)$--sphere centered at $x$ with radius $r$ in $F$ has
  exponentially large $\frac{1}{2}$--avoidant filling volume.  That is, if
  $\alpha_0\from S^{k-1}\to F$ is the sphere of radius $r$ centered at
  $x$, then any extension $\beta_0\from D^k\to X$ whose image avoids
  the ball $B_{r/2}(x)$ has exponentially large volume, i.e.,
  $\vol(\beta_0)\ge e^{c_0r}$.

  To use this result, we find a flat that intersects $Z$ in a sphere.
  Let $F_0$ be a flat opposite to $\fc$ and let $x_0$ be a point in
  $F_0$.  (We can take $x_0$ to be the point on $F_0$ where $h$ is
  maximized, but this is not necessary.)  The boundary of $F_0$ at
  infinity, $(F_0)_\infty$, consists of finitely many chambers, so
  there is some $\rho>0$ such that
  $(F_0)_\infty \subset \cSr_{x_0}(\rho)$.  Choose $r$ so that $r>R$
  and let $x\in X$ be such that $h(x)=r$.  Let $p\in P$ be a group
  element such that $px_0=x$ and let $F=pF_0$.

  We identify $F_\infty$ with $S^{k-1}$ and define
  $\alpha\from S^{k-1}\to Z$ by letting $\alpha(\sigma)=i_x(\sigma)$
  for all $\sigma \in F_\infty$.  By Lemma~\ref{lem:lipschitzPushing},
  this is a Lipschitz map with $\Lip(\alpha)\sim r$.  The image of
  $\alpha$ is the intersection $F\cap Z$, and we claim that $\alpha$
  has exponentially large filling area in $Z$.

  Let $\alpha_0\subset F$ be the sphere centered at $x$ with radius
  $r$, as in Leuzinger's bound.  The spheres $\alpha_0$ and $\alpha$
  both lie in the flat $F$, and since $d(x,Z)=r$, $\alpha_0$ is on the
  inside of $\alpha$.  If $\beta\from D^{k}\to Z$ is an extension of
  $\alpha$, then we can attach an annulus $A$ of volume
  $\vol A\lesssim r^k$ to $\beta$ to construct an extension $\beta_0$
  of $\alpha_0$.  This extension lies outside $B_{r/2}(x)$, so by
  Leuzinger's bound, $\vol \beta_0\gtrsim e^{c_0r}$.  Then
  $\vol \beta=\vol \beta_0-\vol A$, and if $r$ is sufficiently large,
  we have $\vol \beta\ge e^{c_0r/2}$, as desired.
\end{proof}

In order to prove upper bounds, we will need a few more flats.  
As in Section~\ref{sec:shadows}, for $x\in X$, let $C_x\subset X$ be the Weyl
chamber in $X$ based at $x$ and asymptotic to $\fc$.
\begin{lemma}\label{lem:oppositeFlats}
  If $x\in X$ and $\fc$ is a chamber of $X_\infty$, then there is some
  $x'\in C_x$ such that $d(x,x')\lesssim 1$ and some chamber
  $\fd\subset \cSc_{x'}$ such that for all $\fe\in \cSc_{x}$:
  \begin{enumerate}
  \item The chambers $\fe$ and $\fd$ are opposite.
  \item The flat $E_{\fe,\fd}$ is opposite to $\fc$.
  \item The boundary at infinity of $E_{\fe,\fd}$ is contained in $\cSc_{x'}$.
  \end{enumerate}
\end{lemma}

\begin{proof} Without loss of generality we can assume 
  that $x=[e]$ and let $E=E_{x}=[A]$, and let
  $\fc^*=\fc^*_{x}$.

  The chambers in $\cSc_{x}$ are all close to $\fc^*$, so we
  first choose $\fd_0$ so that the flat $E_{\fd_0,\fc^*}$ spanned by
  $\fd_0$ and $\fc^*$ is opposite to $\fc$. To that end let $E_0$ be a flat
  opposite to $\fc$, and let $n\in N$ be such that
  $\fc^* \subset (nE_0)_\infty$.  If $\fd_0$ is the chamber opposite
  to $\fc^*$ in $nE_0$, then $nE_0=E_{\fd_0,\fc^*}$ and it is
  opposite to $n\fc=\fc$.  

  By the above remark, there is an open neighborhood $U$ of $\fc^*$ so
  that for any $\fc'\in \closure(U)$, the flat $E_{\fd_0,\fc'}$ is also opposite
  to $\fc$.  We will use Lemma \ref{lem:dil} to find an element
  $a\in A$ such that $a$ sends $\cSc_{x}$ into $U$.  

  Let $0<r<1$ be such that $\cSc_{x}(r)\subset U$.  Let $\fa$ be
  the Lie algebra of $A$ and let $V\in \fa$ be a unit vector in the
  open Weyl chamber corresponding to $\fc$.  By Lemma~\ref{lem:dil},
  there is a $t>0$ such that $t \lesssim -\log r$ and 
  $$\cSc_{\exp(-tV)}=\exp(-tV) \cSc_{x}\subset \cSc_{x}(r).$$
  Let $a:=\exp(-tV)$ and let $\fd:=a^{-1}\fd_0$.  Then for any
  $\fc'\in \cSc_{x}$, we have $a\fc'\in U$ and thus
  $E_{a\fc',\fd_0}$ is opposite to $\fc$.  It follows that
  $a^{-1}E_{a\fc',\fd_0}=E_{\fc',\fd}$ is opposite to $\fc$.

  Finally, we choose $x'$.  The union
  $$V=\bigcup_{\fc'\in \cSc_{x}}(E_{\fc',\fd})_\infty$$
  is contained in a compact set and consists of chambers opposite to $\fc$, so
  $V\subset \cSc_x(r')$ for some $r'$.  By Lemma~\ref{lem:dil}, there is
  some $x'\in C_x$ such that $d(x,x')\lesssim \log(r'+1)$ and
  $V\subset \cSc_{x'}$.
\end{proof}

\subsection{$(k-2)$--connectivity at infinity}\label{sec:infConnect}

Next, we show that $\xop(\fc)$ is highly connected.  First, we
consider spheres that lie in a single shadow.  Let
$\cSr_{x}=\bigcup_{\fd\in \cSc_{x}}\fd \subset X_\infty$ be the subset of $X_\infty$ covered
by the chambers in  $\cSc_{x}$   as in
Section~\ref{sec:shadows}.  The following lemma is an analogue of
Lemma~4.17 of \cite{Yo}.  Recall that $C_x$ is the Weyl chamber based
at $x$ and asymptotic to $\fc$.
\begin{lemma}\label{lem:connectedShadows}
  If $x\in X$, then there is some
  $x'\in C_x$ such that $d(x,x')\lesssim 1$, $\cSc_{x}\subset
  \cSc_{x'}$, and $\cSr_{x}$ is $(k-2)$--connected
  inside $\cSr_{x'}$.

  That is, if $\alpha\from S^m\to \cSr_{x}$ is Lipschitz and if
  $m\le k-2$, then there is an extension
  $\beta\from D^{m+1}\to \cSr_{x'}$ such that
  $\Lip \beta \lesssim \Lip \alpha+1$.
\end{lemma}
\begin{proof} Let $\cSr_{x}^{(m)}$ be the $m$--skeleton of $ \cSr_{x}$ given by the Tits building structure on $X_\infty$.
  Let $\alpha'\from S^m\to \cSr_{x}^{(m)}$ be a simplicial
  approximation of $\alpha$, and let
  $h_0\from S^m\times [0,1]\to \cSr_x$ be the homotopy from $\alpha$
  to $\alpha'$.
  If $\alpha$ is Lipschitz, we can choose $\alpha'$ so that
  $\Lip \alpha'\lesssim \Lip\alpha$ and choose $h_0$ so that
  $\Lip h_0\lesssim \Lip \alpha+1$.

  We first contract $\alpha'$.  Let $x'$ and $\fd\in \cSc_{x'}$ be as
  in Lemma~\ref{lem:oppositeFlats} and let $u$ be the barycenter of
  $\fd$.  For any $v\in \cSr_x$, there is a flat $E$ that contains $u$
  and $v$ and whose boundary is contained in $\cSc_{x'}$, so the
  Tits geodesic from $u$ to $v$ lies in $\cSr_{x'}$.  Furthermore, if
  $v\in \cSr_x^{(m)}$, then $u$ and $v$ are not opposite to one
  another, so this geodesic is unique.

  Let $h_1\from S^m\times[0,1]\to \cSr_{x'}$ be the map which sends
  $v\times [0,1]$ to the geodesic between $\alpha'(v)$ and $u$.  This
  is a null-homotopy of $\alpha'$, and
  $$\Lip h_1\lesssim 1+\Lip \alpha'.$$
  The constant in the inequality depends on the distance between $u$
  and the $m$--skeleton of $X_\infty$.  By concatenating $h_0$ and
  $h_1$ we obtain a disc $\beta\from D^{m+1}\to \cSr_{x'}$ with
  boundary $\alpha$, and $\Lip \beta \lesssim \Lip \alpha+1$ as
  desired.
\end{proof}

Let $\Delta_Z$ be the infinite-dimensional simplex with vertex set
labeled by $Z$ and let $\langle z_0,\dots, z_k\rangle$ denote the
$k$--simplex with vertices $z_0,\dots,z_k$.  We will use
Lemma~\ref{lem:connectedShadows} to construct a map
$\Omega_\infty\from \Delta_Z \to X_\infty$ that sends each vertex
$\langle v\rangle$ to a direction in the shadow of $v$ and sends each
simplex $\delta$ to a simplex in the shadow of some point $x_\delta$.
If $\delta$ is a simplex of $\Delta_Z$, we denote its set of vertex
labels by $\V(\delta)\subset Z$, so that
$\V(\langle z_0,\dots, z_k\rangle)=\{z_0,\dots, z_k\}$.

\begin{lem}[{see \cite[4.16]{Yo}}] \label{lem:OmegaInfty}
  There is a cellular map
  $$\Omega_\infty\from \Delta_Z^{(k-1)}\to X_\infty,$$
  a constant $c>0$ depending on $X$, and a family of points $x_\delta\in X$,
  where $\delta$ ranges over the simplices of $\Delta_Z^{(k-1)}$.
  This map is $c$--Lipschitz, and for every $\delta$:
  \begin{enumerate}
  \item \label{it:xDeltaDists}
    $d(x_\delta,\V(\delta))\lesssim \diam \V(\delta)+1$ (and consequently,
    $h(x_\delta)\lesssim \diam \V(\delta)+1$).
  \item \label{it:OmegaShadow}$\Omega_\infty(\delta)\subset \cSr_{x_\delta}.$
  \item \label{it:OmegaOrdering} If $\delta'\subset \delta$, then
    $h(x_\delta)\ge h(x_{\delta'})$ and $x_\delta\in D_{x_{\delta'}},$
    where $D_x$ is a neighborhood of the chamber $C_x$ as in
    Section~\ref{sec:shadows}.
  \item $h(x_\delta)\ge 1$.
  \end{enumerate}
\end{lem}
\begin{proof}
  We will construct $\Omega_\infty$ one dimension at a time using
  Lemma~\ref{lem:connectedShadows}.  Let $\tau_0$ be the barycenter of
  $\fc$ and let $\theta:=\angle(\tau_0,\tau)$.  By the standing
  assumptions, we have $\theta<\pi/2$.
  For $x\in X$, let $r_x\from [0,\infty)\to X$ be the geodesic ray
  based at $x$ and asymptotic to $\tau_0$, so that
  $$h(r_x(t))=h(x)+t\cos \theta$$ 
  for all $t$.

  We start by defining $\Omega_\infty$ on the vertices of $\Delta_Z$.
  For $z\in Z$, let $b_z$ be the barycenter of the chamber $\fc^*_z$
  and let
  $$\Omega_\infty(\langle z\rangle)=b_z$$
  and let $x_{\langle z\rangle}=r_z(\sec \theta)$.  This satisfies all
  of the desired conditions.

  Now, suppose by induction that $0<m\le k-1$, that we have defined
  $\Omega_\infty$ on $\Delta_Z^{(m-1)}$, and that we have defined
  $x_{\delta'}$ for every simplex $\delta'$ with $\dim \delta'<m$.
  Let $\delta$ be an $m$--simplex and let $z\in Z$ be one of its
  vertices.

  By part \ref{it:xDeltaDists} of the lemma, the points $x_{\delta'}$
  are contained in a ball $B_{z}(R)$ with $R\sim \diam \V+1$.  By
  Lemma~\ref{lem:largeShadows}, there is a $t\lesssim R+1$ such that
  if $x_0=r_{z}(t)$ and $\delta'$ is a face of $\delta$, then
  $\cSc_{x_{\delta'}}\subset \cSc_{x_0}$ and
  $x_0\in D_{x_{\delta'}}$.

  By part \ref{it:OmegaShadow} and Lemma~\ref{lem:connectedShadows},
  there is an $x'\in C_{x_0}$ such that
  $\Omega_\infty|_{\partial \delta}$ is contractible in $\cSr_{x'}$
  and $d(x_0,x')\lesssim 1$.  We let $x_\delta=x'$ and define the
  extension $\Omega_\infty|_{\delta}\from \delta\to \cSr_{x'}$ using
  Lemma~\ref{lem:connectedShadows}.  The desired properties of
  $x_\delta$ and $\Omega_\infty|_{\delta}$ are easy to check.
\end{proof}

\subsection{Lipschitz connectivity of $Z$}\label{sec:LipConnZ}

Finally, we show that $Z$ is Lipschitz $(k-2)$--connected.  Our main
tool is a lemma similar to Lemma~3.2 of \cite{Yo}.  
\begin{lemma}\label{lem:LipConn}
  Suppose that $Z\subset X$, that $X$ is Lipschitz $(k-2)$--connected
  and that for any $r$, there is a Lipschitz retraction
  $R_r\from N_r(Z)\to Z$, where $N_r(Z)=\{x\in X\mid d(x,Z)<r\}$.
  Then, if there is a map $\Omega\from \Delta_Z^{(k-1)}\to Z$ such
  that
  \begin{enumerate}
  \item for all $z\in Z$, $d(\Omega(\langle z\rangle),z)\lesssim 1$, and
  \item for any simplex $\delta\subset \Delta_Z$, we have
    $$\Lip \Omega|_{\delta}\lesssim \diam \V(\delta)+1$$
  \end{enumerate}
  then $Z$ is Lipschitz $(k-2)$--connected.  
\end{lemma}
\begin{proof}
  The proof is very similar to the proof of Lemma~3.2 of \cite{Yo},
  which constructs a Lipschitz extension using a Whitney
  decomposition.  The main difference is that we assume that
  $\Lip \Omega|_{\delta}\lesssim \diam \V(\delta)+1$ rather than
  $\Lip \Omega|_{\delta}\lesssim \diam \V(\delta)$.  The weaker
  inequality means that the Lipschitz constant of $\Omega$ on small
  simplices can be unbounded, so we use the Lipschitz connectivity of
  $X$ to extend the map near the boundary.

  Let $L>0$, let $D^{k-1}(L)\subset \R^{k-1}$ be the closed ball of
  radius $L$, and let $S^{k-2}(L):=\partial D^{k-1}(L)$.  It suffices to
  show that for any $L>0$ and any $1$--Lipschitz map
  $\alpha\from S^{k-2}(L) \to Z$, there is an extension
  $\beta\from D^{k-1}(L)\to Z$ such that $\Lip(\beta)\lesssim 1$.

  Let $\alpha$ be such a map.  By the Whitney covering lemma, we can
  decompose the interior of $D:=D^{k-1}(L)$ into a countable union of
  dyadic cubes with disjoint interiors such that
  $\diam C\sim d(C,\partial D)$ for each cube $C$.  The barycentric
  subdivision of this cover is a triangulation $T$ of the interior
  of $D$ such that each simplex is bilipschitz equivalent to a scaling
  of the standard simplex.  Let $c>1$ be such that for each simplex
  $\sigma$ with $\dim \sigma >0$, we have
  $$c^{-1} d(\sigma,\partial D)\le \diam \sigma\le c d(\sigma,\partial D).$$

  Define $h\from T^{(0)}\to \partial D$ so that for all
  $v\in T^{(0)}$, we have $d(v,h(v))=d(v,\partial D)$.  For each edge
  $e=(v,w)$, we have 
  $$d(h(v),h(w))\le d(v,\partial D)+d(v,w)+d(w,\partial D).$$
  Since $d(e,\partial D) \le c \diam e$, this implies that
  $d(h(v),h(w))\lesssim d(v,w)$ and thus $\Lip(h)\lesssim 1$.  

  Define
  \begin{equation*}
    \beta_0(v):=\begin{cases}
      \Omega(\langle \alpha(h(v))\rangle) & \text{ if $d(v,\partial D)\ge c^{-1}$ }\\
      \alpha(h(v)) & \text{ otherwise. }\\
    \end{cases}
  \end{equation*}
  If $e=(v,w)$ is an edge of $T$ such that $\ell(e)<c^{-2}/2$, then
  $d(v,\partial D)<c^{-1}$ and $d(w,\partial D)<c^{-1}$.  Then
  $$d(\beta_0(v),\beta_0(w))=d(\alpha(h(v)),\alpha(h(w)))\le
  \Lip(\alpha)\Lip(h)d(v,w)\lesssim d(v,w).$$
  If $e=(v,w)$ is an edge of $T$ such that $\ell(e)\ge c^{-2}/2$,
  then $\Omega$ may introduce bounded error:
  $$d(\beta_0(v),\beta_0(w))\lesssim d(\alpha(h(v)),\alpha(h(w)))+1\lesssim d(v,w).$$
  It follows that $\beta_0$ is a Lipschitz map with
  $\Lip(\beta_0)\sim 1$.

  If $\sigma=\langle v_0,\dots,v_k\rangle$ is a simplex of $T$ with
  diameter at least 1, then $d(v_i,\partial D)\ge c^{-1}$, so we may
  extend $\beta$ to $\sigma$ so that it sends $\sigma$ to
  $\Omega(\langle g_0(v_0),\dots,g_0(v_k)\rangle)$.  This extension is
  Lipschitz with bounded Lipschitz constant, and it remains to extend
  $\beta$ to simplices with diameter less than 1.

  We work one dimension at a time.  We have already defined $\beta$ on
  the 0--skeleton of $T$ in a Lipschitz fashion.  Suppose that
  $d\le k-2$, that $\beta$ has been defined on $T^{(d-1)}$, and that
  there is a $c_{d-1}>0$, independent of $\alpha$, such that
  $\Lip \beta|_{T^{(d-1)}}\le c_{d-1}$.  We claim that there is an
  extension of $\beta$ to $T^{(d)}$ and a $c_d$ independent of
  $\alpha$ such that $\Lip \beta|_{T^{(d)}}\le c_d$.

  Suppose that $\sigma$ is a $d$--simplex of $T$ such that
  $\diam \sigma<1$.  Then
  $\beta|_{\partial \sigma}\from \partial \sigma \to Z$ is a map with
  $\Lip \beta|_{\partial \sigma}\le c_{d-1}$.  Since $\sigma$ is
  bilipschitz equivalent to a scaling of the standard simplex, we can
  identify $\partial \sigma$ with a scaled version of $S^{d-1}$ in a
  bilipschitz way and use the Lipschitz connectivity of $X$ to produce
  an extension $f\from \sigma\to X$ with Lipschitz constant
  $\Lip f\lesssim c_{d-1}$.  Then there is a $r_{d}$ such that
  $f(\sigma)\subset N_{r_{d}}(Z)$, and if
  $\beta|_\sigma=R_{r_{d}}\circ f$, then
  $$\Lip \beta|_\sigma\lesssim c_{d-1}\Lip R_{r_{d}},$$ 
  as desired.  Repeating this process for each small simplex in
  $T^{(d)}$, we obtain the desired extension.
\end{proof}

We use $\Omega_\infty$ and the map $i_u$ constructed in
Lemma~\ref{lem:lipschitzPushing} to construct an $\Omega$ that
satisfies Lemma~\ref{lem:LipConn}.  

\begin{lem}\label{lem:Omega}
  There is a map $\Omega\from \Delta_Z^{(k-1)}\to Z$ satisfying the
  conditions of Lemma~\ref{lem:LipConn}.  Consequently, $Z$ is
  Lipschitz $(k-2)$--connected.
\end{lem}
\begin{proof}
  For $\rho>0$, let
  $$Y(\rho):=\{(\sigma,x)\in \xop(\fc)\times X \mid h(x)\ge 1,\sigma \in
  \cSr_{x}(\rho)\}.$$
  We give $Y(\rho)$ the metric
  $$d_Y((\sigma_1,x_1),(\sigma_2,x_2)):=d(x_1,x_2)+\min\{h(x_1),h(x_2)\}\cdot \angle(\sigma_1,\sigma_2).$$
  In Lemma~\ref{lem:lipschitzPushing}, we showed that the map
  $I\from Y(\rho) \to X$ given by $I(v,x)=i_x(v)$ is Lipschitz with
  Lipschitz constant depending on $\rho$.  The map $\Omega$ will be a
  composition $\Omega=I\circ W$, where
  $W\from \Delta_Z^{(k-1)}\to Y(\rho)$ is a map based on the map
  $\Omega_\infty$ and the points $x_\delta$ constructed in
  Lemma~\ref{lem:OmegaInfty}.

  \begin{figure}
    \begin{center}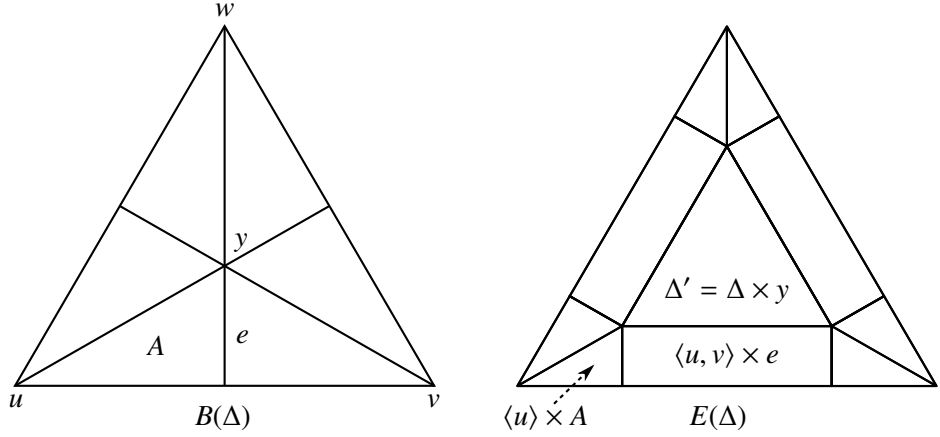
    \end{center}
    \caption{\label{fig:triangles} Each cell of the ``exploded
      simplex'' $E(\delta)$ is a product of a cell of $\delta$ and a
      cell of $B(\delta)$.}
  \end{figure} 

  To construct $W$, we use the exploded simplices used in \cite{Yo}.
  If $\Delta$ is a simplex, the exploded simplex $E(\Delta)$ is a
  cellular subdivision of a simplex $\Delta$ with the following
  properties (see Figure~\ref{fig:triangles}):
  \begin{enumerate}
  \item $E(\Delta)$ contains a copy $\Delta'$ of $\Delta$ at its
    center.
  \item $E(\Delta)$ subdivides each face of $\Delta$ into an exploded
    simplex of lower dimension.
  \item Each cell of $E(\Delta)$ is of the form
    $\Delta_1\times \Delta_2$, where $\Delta_1$ is a face of $\Delta$
    and $\Delta_2$ is a face of the barycentric subdivision
    $B(\Delta)$.

    Specifically, recall that the vertex set of $B(\Delta)$ is the set
    of faces of $\Delta$.  If $\delta_0,\dots, \delta_k$ are faces of
    $\Delta$ that form a flag---that is, if
    $\delta_0\subset \dots\subset \delta_k$---then
    $\langle \delta_0,\dots,\delta_k\rangle$ is a simplex of $B(\Delta)$.
    Each cell of $E(\Delta)$ is of the form
    $$\delta \times \langle \delta_0,\dots,\delta_k\rangle$$
    for some flag $\delta_0,\dots,\delta_k$ and some face $\delta$ of
    $\delta_0$.  
  \item Since each cell of $E(\Delta)$ is a cell of
    $\Delta\times B(\Delta)$, we can define maps
    $p_1\from E(\Delta)\to \Delta$ and
    $p_2\from E(\Delta)\to B(\Delta)$ ($\rho_1$ and $\rho_2$ in
    \cite{Yo}) coming from the projections to the first and second
    factors.  These maps are Lipschitz.  The map $p_1$ expands the
    central simplex to cover $\Delta$ and shrinks the collar to the
    boundary.  The map $p_2$ collapses the central simplex to the
    barycenter of $\Delta$, sends the central simplices of all the
    faces to the corresponding barycenters, and sends the collar
    surjectively to $\Delta$.
  \end{enumerate}

  If $Q$ is a simplicial complex, we can form a cellular subdivision
  $E(Q)$ by exploding each simplex.  The maps $p_1$ and $p_2$ on
  each simplex agree on overlaps, so we combine them to form maps
  $p_1\from E(Q)\to Q$ and $p_2\from E(Q)\to B(Q)$ defined on
  all of $E(Q)$.

  By convention, we will write the vertices of a simplex of $B(Q)$ in
  ascending order, so if $\langle \delta_0,\dots, \delta_d\rangle$ is
  a simplex of $B(Q)$, then the $\delta_i$'s are simplices of $Q$ and
  $\delta_0\subset \dots \subset \delta_d$.

  Let $Q=\Delta_Z^{(k-1)}$.  For each simplex $\delta$ of $\Omega$, we
  define $\Omega$ on $\delta'$ by 
  \begin{equation}\label{eq:OmegaCenters}
    \Omega(\delta')=i_{x_\delta}(\Omega_\infty(\delta)),
  \end{equation}
  where $\delta'$ is the central simplex of $E(\delta)$.  We will extend
  $\Omega$ to the collars of the $E(\delta)$'s by using the
  projections $p_1$ and $p_2$. 

  Specifically, we will define a map $F\from B(Q)\to X$ and
  let $\Omega=I\circ W$, where
  $$W=(\Omega_\infty \circ p_1, F \circ p_2).$$
  The map $F$ will satisfy:
  \begin{enumerate}
  \item\label{it:centersToXd}
    The complex $B(Q)$ has a vertex $b_\delta$ at the barycenter
    of each simplex $\delta$ of $Q$.  For all $\delta$,
    $F(b_\delta)=x_\delta$.  (This ensures that the extension agrees
    with the map defined in equation \eqref{eq:OmegaCenters}.)
  \item If $\Delta=\langle \delta_0,\dots,\delta_d\rangle\subset
    B(Q)$, then $\Lip F|_\Delta\lesssim \diam \V(\delta_d)+1$.  
  \item There is a $\rho>0$ such that if
    $$y\in \langle \delta_0,\dots,\delta_d\rangle\subset B(Q),$$
    then $\cSc_{F(y)}(\rho)\supset \cSc_{x_{\delta_0}}$.
  \end{enumerate}

  We define $F$ one dimension at a time.  For each vertex $v$ of $Q$,
  we define $F(v)=x_v$.  If $\delta$ is a simplex of $Q$ and we have
  already defined $F$ on $\partial \delta$, we extend $F$ on the rest
  of $\delta$ by coning it off to $x_\delta$.  That is, every point in
  $\delta$ is on a line segment between $b_\delta$ and a point
  $y\in \partial \delta$.  We send $b_\delta$ to the point $x_\delta$
  and we send each such segment to a geodesic segment from $x_\delta$
  to $F(y)$.  

  This satisfies condition 1 by construction.  Condition 2 follows
  from the fact that $X$ is CAT(0) and
  Lemma~\ref{lem:OmegaInfty}.\ref{it:xDeltaDists}.  It only remains to
  check condition 3.  Suppose that
  $\Delta=\langle \delta_0,\dots,\delta_d\rangle\subset B(Q)$ and that
  $y\in \Delta$.  For $i=0,\dots,d$, let $x_i=x_{\delta_i}$.  Let
  $D_0=D_{x_0}$ be a neighborhood of $C_{x_0}$ as in
  Section~\ref{sec:shadows}.  By
  Lemma~\ref{lem:OmegaInfty}.\ref{it:OmegaOrdering}, for all
  $i=0,\dots, d$, we have $x_i\in D_0$.  Since $C_{x_0}$ is convex,
  $D_0$ is convex.  Since $F(\Delta)$ is contained in the convex hull
  of the $x_i$, we have $F(y)\in D_0$.  By Lemma~\ref{lem:DxShadows},
  there is a $\rho$ depending on $X$ such that
  $\cSc_{x}\subset \cSc_{F(y)}(\rho)$, as desired.

  It follows that the image of $W$ lies in $Y(\rho)$.  Suppose that
  $q\in Q$ and let
  $\delta \times \langle \delta_0,\dots,\delta_k\rangle$ be a cell of
  $E(Q)$ containing $q$.  Here, $\delta$ is a face of $Q$ and
  $\delta\subset \delta_0\subset \dots\subset \delta_k.$
  Then $p_1(q)\in \delta$ and
  $p_2(q)\in \langle \delta_0,\dots,\delta_k\rangle$.  By
  Lemma~\ref{lem:OmegaInfty}.\ref{it:OmegaShadow}, we have
  $$\Omega_\infty(p_1(q))\in \cSr_{x_\delta}\subset
  \cSr_{x_{\delta_0}}.$$
  By property 3 of $F$, we have $\cSr_{x_{\delta_0}}\subset
  \cSr_{F(p_2(q))}(\rho)$, so
  $$W(q)=(\Omega_\infty(p_1(q)),F(p_2(q)))\in Y(\rho).$$
  We may thus define $\Omega=I\circ W$.

  Finally, we claim that $\Omega$
  satisfies the conditions in Lemma~\ref{lem:LipConn}.  First, if
  $z\in Z$, let $v=\langle z \rangle$ be the corresponding vertex of
  $\Delta_Z$.  Then
  $$\Omega(v)= I(\Omega_\infty(v),F(v))=i_{x_z}(\sigma),$$
  where $\sigma=\Omega_\infty(v)\in \cSr_{x_z}(\rho)$, and
  \begin{align*}
    d(z,\Omega(v))&\le d(z,x_z)+d(x_z,i_{x_z}(\sigma))\\ 
                  &\lesssim1+h(x_z)+\rho \lesssim 1
  \end{align*}
  by Lemmas~\ref{lem:lipschitzPushing} and \ref{lem:OmegaInfty}.  This
  proves property 1 of Lemma~\ref{lem:LipConn}.

  If $\delta\subset \Delta_Z$ and if $q_1,q_2\in \delta$, properties 1
  and 2 of $F$ imply that 
  $$h(F(p_2(q_i)))\lesssim \diam \V(\delta)+1.$$
  It follows that
  \begin{align*}
    d(\Omega(q_1),\Omega(q_2))
    &\le (\Lip I)d_Y(W(q_1),W(q_2))\\ 
    &\lesssim d(F(p_2(q_1)),F(p_2(q_2)))+(\diam\V(\delta)+1) \cdot
      \angle\bigl(\Omega_\infty(p_1(q_1)),\Omega_\infty(p_1(q_2))\bigr)\\
    & \lesssim (\diam\V(\delta)+1)(d(q_1,q_2),
  \end{align*}
  implying property 2.
\end{proof}

Lemma~\ref{lem:LipConn} then implies that $Z$ is Lipschitz
$(k-2)$--connected.  This concludes the proof.

\section{Proof of Theorem A}

We use the following result proved in \cite{Le1}, Thm. 3.6 (see also
\cite{Bo}, \S 13).

\begin{prop} 
  Let $\Gamma$ be an arithmetic lattice of $\mathbb Q$--rank $1$ in a
  linear, semisimple Lie group $G$ and let $X=G/K$ be the associated symmetric
  space. Then any orbit of $\Gamma$ in $X$ is quasi-isometric to a set
  $Y:=X\setminus \bigcup_i B_i$, where the $B_i$ comprise a countable
  set of disjoint horoballs.
\end{prop} 

We can write $X$ as a Riemannian product of irreducible symmetric
spaces, $X=\prod_{i=1}^mX_i$ (corresponding to the decomposition
$G=\prod_{i=1}^mG_i$ of $G$ into simple factors). The boundary  of $X$ at infinity is the spherical join of the boundaries of the factors.

Assume now (as in Theorem A) that the lattice $\Gamma$ is
\emph{irreducible}.  We claim that none of the centers of the
horoballs in the above proposition are contained in a proper join factor of
$X_\infty$.
By way of contradiction, assume that $m\geq 2$ and that one of the horoballs, say $B$, is
centered in a join factor associated to, say, $X_1$.  Then $B$ is a
sublevel set of the Busemann function associated to a geodesic in $X$
of the form $(c_1(t), p_2,\ldots,p_m)$.  Thus $B$ has the form
$B=B_1\times \prod_{i=2}^mX_i$, where $B_1$ is a horoball in
$X_1$. The projection of a $\Gamma$--orbit on the first factor $X_1$
then avoids $B_1$. In particular, the projection of $\Gamma$ to the
first factor $G_1\cong G/\prod_{i=2}^mG_i$ cannot be dense in $G_1$.  This contradicts
irreducibility, see \cite{Ra}, Cor. 5.21 (5).

By Theorem~B, the boundary of each horoball $B_i$ is Lipschitz
$(k-2)$--connected, where $k=\rank X$.  By part 2 of
Proposition~\ref{prop:LipUndistorted}, the  subspace $Y\subset $ is undistorted up to
dimension $k-1$.  The claim on its isoperimetric inequalities is a
consequence of \cite{Le2}, which asserts that a symmetric space $X$
satisfies Euclidean isopermetric inequalites below the rank.

Finally, the lower bound in Theorem~A follows from
Proposition~\ref{prop:lowerBoundsHorospheres}.  By the proposition,
for each $i$ and for all sufficiently large $r$, there is a Lipschitz
sphere $\alpha\from S^{k-1}\to \partial B_i$ such that
$\Lip(\alpha)\sim r$ and such that any Lipschitz extension
$\beta\from D^{k}\to \partial B_i$ satisfies $\vol \beta\ge e^{cr}.$

Let $p\from Y\to \partial B_i$ be the nearest-point projection; since
$B_i$ is convex, this is a distance-decreasing map.  If
$\beta'\from D^{k}\to Y$ is an extension of $\alpha$, then $p\circ
\beta'\from D^k\to \partial B_i$ is also an extension, and
$$\vol \beta\ge \vol \beta'\ge e^{cr}.$$

\noindent{\tt enrico.leuzinger@kit.edu}

\noindent\textsc{Institute of Algebra and Geometry\\
 Karlsruhe Institute of technology (KIT), 76131 Karlsruhe, Germany}

\vspace{1ex}

\noindent{\tt ryoung@cims.nyu.edu}

\noindent\textsc{Courant Institute of mathematical sciences\\New York University, NY 10012, USA}
\end{document}